%% file: Confetti.tex
\documentclass[11pt]{article}
  \usepackage{a4wide}
  \usepackage{amsmath}
  \usepackage{amssymb}
  \usepackage{latexsym}
  \usepackage[pdftex]{graphicx}
  \usepackage{color}
  \usepackage[mathscr]{euscript}
  \usepackage[section] {placeins}
  \usepackage{hyperref}

  \newenvironment{proof}{\vspace{1ex}\noindent{\bf Proof:}}{\hspace*{\fill}$\blacksquare$\vspace{1ex}}
  \newenvironment{proofof}[1]{\vspace{1ex}\noindent{\bf Proof of #1:}}{\hspace*{\fill}$\blacksquare$\vspace{1ex}}

  \newtheorem{theorem}{Theorem}[section]
  \newtheorem{lemma} [theorem] {Lemma}
  \newtheorem{corollary} [theorem] {Corollary}
  \newtheorem{claim} [theorem] {Claim}
  \newtheorem{proposition} [theorem] {Proposition}

\newcommand{\Bcal}[0]{\ensuremath{{\mathcal B}}}
\newcommand{\Ccal}[0]{\ensuremath{{\mathcal C}}}

\newcommand{\Ical}[0]{\ensuremath{{\mathcal I}}}

\newcommand{\Kcal}[0]{\ensuremath{{\mathcal K}}}

\newcommand{\Pcal}[0]{\ensuremath{{\mathcal P}}}

\newcommand{\eR}[0]{\ensuremath{ \mathbb R}}

\newcommand{\eN}[0]{\ensuremath{ \mathbb N}}

\newcommand{\Ascr}[0]{\ensuremath{{\mathscr A}}}

\newcommand{\Cscr}[0]{\ensuremath{{\mathscr C}}}

\newcommand{\norm}[1]{\ensuremath{\|#1\|}}

\newcommand{\Pee}[0]{\ensuremath{{\mathbb P}}}
\newcommand{\Ee}[0]{\ensuremath{{\mathbb E}}}

\newcommand{\isd}[0]{\hspace{.2ex} \raisebox{-.1ex}{$=$} \hspace{-1.5ex} 
\raisebox{1ex}{{$\scriptstyle d$}} \hspace{.8ex} }

 \newcommand{\eps}{\varepsilon}

\DeclareMathOperator{\diam}{diam}

\DeclareMathOperator{\inter}{int}

\DeclareMathOperator{\Po}{Po}

\DeclareMathOperator{\Be}{Be}
\DeclareMathOperator{\dd}{d}
\DeclareMathOperator{\Var}{Var}
\newcommand{\lb}[0]{\ensuremath{\lambda_{\rm{b}}}}
\newcommand{\lw}[0]{\ensuremath{\lambda_{\rm{w}}}}

\newcommand{\Pb}[0]{\ensuremath{\Pcal_{\rm{b}}}}
\newcommand{\Pw}[0]{\ensuremath{\Pcal_{\rm{w}}}}
\newcommand{\ob}[0]{\ensuremath{\omega_{\rm{b}}}}
\newcommand{\ow}[0]{\ensuremath{\omega_{\rm{w}}}}
\newcommand{\opt}[0]{\ensuremath{\overline{p}(t)}}
\newcommand{\op}[0]{\ensuremath{\overline{p}}}
\newcommand{\lf}[0]{\ensuremath{\text{left}}}
\newcommand{\ri}[0]{\ensuremath{\text{right}}}
\newcommand{\tp}[0]{\ensuremath{\text{top}}}
\newcommand{\bo}[0]{\ensuremath{\text{bottom}}}
\newcommand{\Ax}[1]{{\bf(\Ascr-#1)}}
\newcommand{\CC}[1]{{\bf(\Cscr-#1)}}

 \title{The critical probability for confetti percolation equals $1/2$}
  \author{Tobias M\"uller\thanks{Utrecht University, Utrecht, the Netherlands. E-mail: \texttt{t.muller@uu.nl}. 
Part of the work in this paper was done while this author was supported by a VENI grant from Netherlands Organisation for Scientific Research (NWO).}}
  
\begin{document}

  \maketitle

\begin{abstract}
In the confetti percolation model, or two-coloured dead leaves model, radius one disks arrive on
the plane according to a space-time Poisson process. 
Each disk is coloured black with probability $p$ and white with probability $1-p$.
In this paper we show that the critical probability for confetti percolation equals $1/2$.
That is, if $p>1/2$ then a.s.~there is an unbounded curve in the plane all of whose points are black; while
if $p \leq 1/2$ then a.s.~all connected components of the set of black points are bounded.
This answers a question of Benjamini and Schramm~\cite{BenjaminiSchramm98}.
The proof builds on earlier work by Hirsch~\cite{HirschArxiv} and makes use of an adaptation of 
a sharp thresholds result of Bourgain.
\end{abstract}


\section{Introduction and statement of results}

The confetti percolation model is informally described as follows.
Imagine that disks of equal radius have been raining down on the plane for a very long time. Each disk is either black (with probability $p$) 
or white (with probability $1-p$).
Suddenly the rain of confetti disks stops and we examine the pattern of colours that we see on the ground. 
Here the colour of a point of the plane is of course determined by the disk that was last to arrive among all disks that cover the point.

\begin{figure}[h]
\begin{center}
\includegraphics[width=5cm]{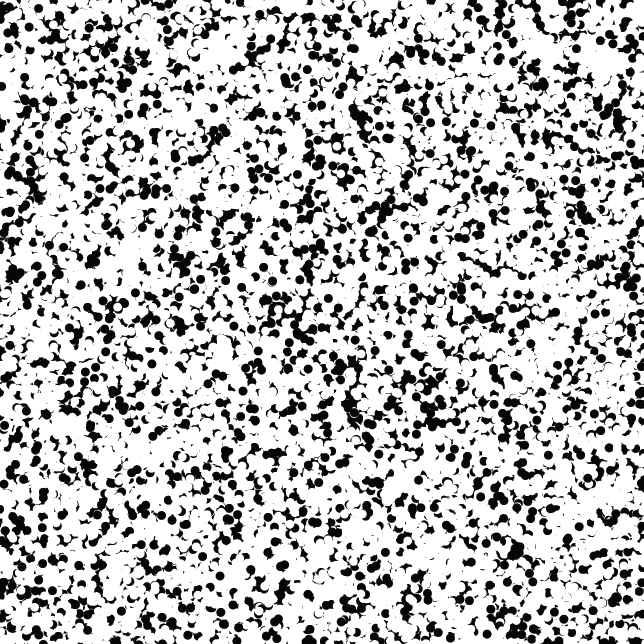}\hspace{.2cm}%
\includegraphics[width=5cm]{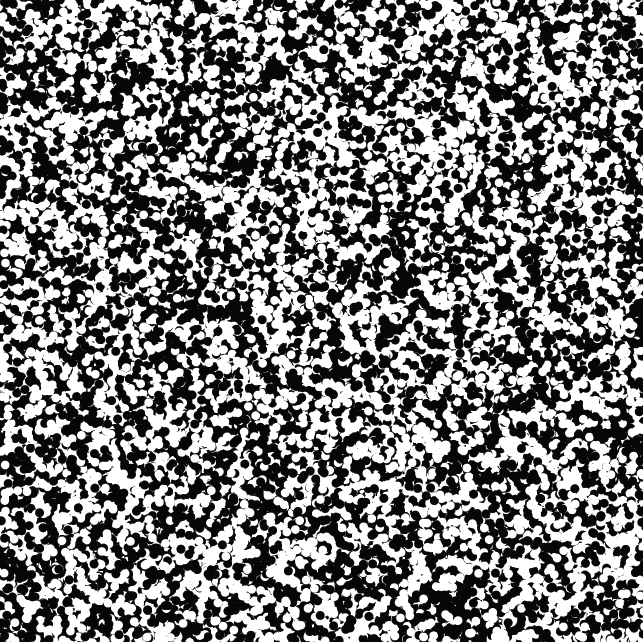}\hspace{.2cm}%
\includegraphics[width=5cm]{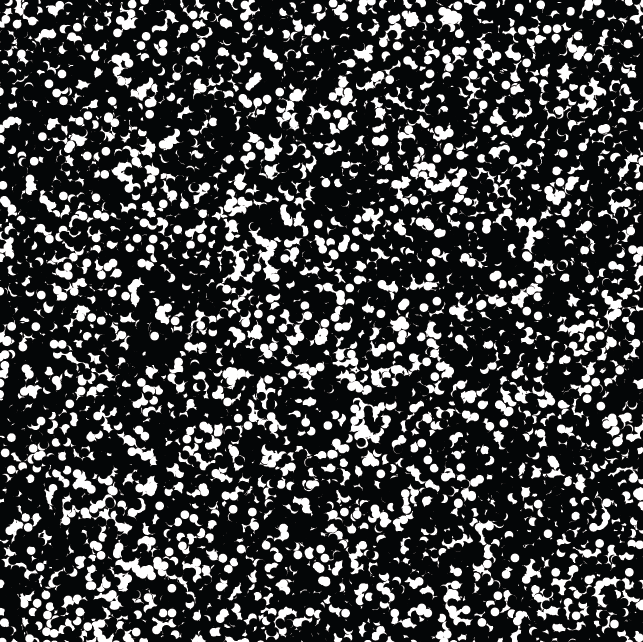}%
\vspace{-.5cm}
\end{center}
\caption{Simulations of confetti percolation with $p=\tfrac{1}{4}, \tfrac12, \tfrac34$. 
A square of dimensions $200\times200$ is shown.\label{fig:simulation}}
\end{figure}

A more formal (and precise) definition of the confetti percolation model is as follows.
We start with a Poisson process $\Pcal$ of constant intensity $\lambda > 0$ on 
$\eR^2 \times (-\infty, 0]$. Around each point of $\Pcal$ we center a closed horizontal disk of radius one. We colour each of these disks black 
with probability $p$ and white with probability $1-p$, independently of the colours of all other disks and of $\Pcal$.
To determine the colour of a point $q \in \eR^2$, we draw a vertical line $\ell$ through $q$ (here and in the rest of the paper we identify 
$\eR^2$ with $\{z = 0\} \subseteq \eR^3$) and assign to $p$ the colour of the highest disk that intersects the line $\ell$.
We can think of the $z$-coordinate of a point of $\Pcal$ as the time when the corresponding confetti disk arrives, obscuring parts
of pre-existing confetti disks from view.
The confetti model is a special case of the {\em colour dead leaves model} introduced by Jeulin~\cite{Jeulin1996} for the purpose
of simulating mineral structures. The model of Jeulin allows for more colours and different shapes of the confettis (leaves).

We say that {\em percolation} occurs if there exists an unbounded curve $\gamma \subseteq \eR^2$ all of whose points are black.
As usual, the {\em critical probability} is defined as

\[ 
p_c := \inf \left\{ p : \Pee_p( \text{percolation} ) > 0\right\}. 
\]

\noindent
We index the probability only by $p$ and not by $\lambda$ since the precise value of $\lambda$ is irrelevant -- see the next 
section for a detailed explanation.
Benjamini and Schramm~\cite{BenjaminiSchramm98} asked whether $p_c = 1/2$.
Here we answer their question in the affirmative.

\begin{theorem}\label{thm:main}
$p_c = 1/2$.
\end{theorem}

Very recently, Hirsch~\cite{HirschArxiv} proved a version of Theorem~\ref{thm:main} for the case when instead of disks, squares are used as 
the confetti. 
Part of Hirsch's arguments in fact work for a wide range of shapes. 
In particular, the fact that $p_c \geq 1/2$, also in the setting with disk-shaped confetti, is already proved by Hirsch.
However, technical issues forced Hirsch to restrict himself to the case of squares in his proof of the full result. 
He also asked for generalizations of his result to more general shapes.
Our proof of Theorem~\ref{thm:main} does in fact work for a large class of shapes. 
For the sake of the exposition we will focus on disk shaped confettis and we sketch the adaptations that need to be made
to generalise the proof later, in Section~\ref{sec:othershapes}.
A crucial step in our approach is the application of an asymmetric version of a powerful ``sharp threshold'' result of Bourgain 
(that appeared in the appendix to Friedgut's paper~\cite{Friedgut99}). We believe similar arguments should work in many other
percolation models.

\medskip

{\bf Overview of the paper and the main ideas in the proof.}
With the work that has already been done by Hirsch~\cite{HirschArxiv}, all that remains for us to prove is that 
percolation does occur almost surely when $p>1/2$.
By standard machinery in percolation theory, it in fact suffices to show that, when $p>1/2$, a rectangle 
of dimensions $3s\times s$ has a black crossing in the long direction with probability that 
will get arbitrarily close to one as we send $s$ to infinity (and $p>1/2$ stays fixed).
To achieve this, we first show that we can approximate such a crossing event by a discrete event defined in
terms of finitely many Bernoulli random variables. To define this discrete event, we dissect a relevant part of $\eR^2 \times (-\infty,0]$ 
into small, equal-sized cubes. Each Bernoulli random variable indicates whether or not there is a black, respectively white, point of the 
Poisson process inside a given cube. These Bernoulli random variables are independent and their means take one of 
two values, depending on whether the random variable detects black or white points. 
A powerful tool by Bourgain (that appeared in the appendix to Friedgut's paper~\cite{Friedgut99}) gives a condition which must hold
if a monotone event defined in terms of i.i.d.~Bernoulli random variables does not have a rapid transition from probability nearly zero to 
probability nearly one, as we vary the common mean of the Bernoullis from zero to one. 
Roughly speaking, it says that if there is no such rapid transition and the parameters are chosen such that the probability of our monotone 
event is neither too small nor too large, then
there must be a bounded number of variables such that the probability of the event, conditioned on those variables all equalling one, is 
close to one.
Proposition~\ref{prop:BourTobias} below generalizes this result to the case where the Bernoulli random variables are independent but 
may have different means. We use it show that if the probability of the (discrete approximations to our) crossing events does not undergo 
such a sharp increase at $p=1/2$, then there is a bounded number of cubes such that the status of those cubes
influences the crossing events noticeably. That is, conditioning on black/white points in these
cubes increases/decreases the crossing probability by a constant.
This would mean that with constant (unconditional) probability it holds that (a) every crossing gets close to the projection of at least 
one of these cubes on the plane and (b) there is at least one crossing. 
We will see that this is impossible, and hence that there must be a rapid transition 
for crossing probabilities at $p=1/2$.
Together with standard percolation machinery this gives that percolation does occur almost surely when $p>1/2$.

What distinguishes our approach from that of Hirsch~\cite{HirschArxiv} is that his approach, which follows that of 
Bollob\'as and Riordan~\cite{BollobasRiordanVoronoi}, relied on a result of Friedgut and Kalai~\cite{FriedgutKalai96}
to show that there is a ``sharp threshold'' for crossing probabilities, while we instead will use Proposition~\ref{prop:BourTobias}
to achieve the same.
Like Proposition~\ref{prop:BourTobias}, the result of Friedgut and Kalai applies to monotone events defined in terms of 
independent Bernoulli random variables, but in order for it to imply a sharp threshold it is needed that the common mean of these
Bernouilli random variables is not too small. This meant that the discretizations Hirsch used could not be arbitrarily fine, which 
in turn led to considerable technical difficulties. In contrast, our use of Proposition~\ref{prop:BourTobias} does allow us to 
use arbitrarily small cubes in our discretizations. 

In the next section, we provide some preliminary discussion and results that we will need in the proof of Theorem~\ref{thm:main}. 
In Section~\ref{sec:DiscrAppr} we define and formally justify the discrete approximations to the box-crossing events.
Section~\ref{sec:pGreaterHalf} contains the main part of our argument, which applies Proposition~\ref{prop:BourTobias} to the 
discrete approximations to crossing events. In Section~\ref{sec:grandfinale} we direct the reader to a place in the literature where the standard 
argument that completes the proof of Theorem~\ref{thm:main} can be found.
Section~\ref{sec:othershapes} briefly sketches the changes that need to be made to adapt the proof to work in the case of other confetti shapes 
besides the unit disk. The proof of Proposition~\ref{prop:BourTobias} can be found in Appendix~\ref{sec:BourApp}.

\section{Notation and preliminaries\label{sec:preliminaries}}

Throughout this paper, $\Po(\lambda)$ will denote the Poisson distribution
with parameter $\lambda$ and $\Be(p)$ will denote the Bernoulli distribution with parameter $p$.

A subset $A$ of the discrete hypercube $\{0,1\}^n$ is called an {\em up-set} if
it is closed under increasing coordinates. 
That is, whenever we take a point of $A$ and we change one of its coordinates into
a one, then the resulting point is still in $A$.
That is, if $\overline{a} = (a_1, \dots, a_n) \in A, \overline{b} = (b_1,\dots, b_n) \in \{0,1\}^n$ and 
$a_i \leq b_i$ for all $i$ then also $b \in A$. 

For $\overline{p} = (p_1,\dots, p_n) \in (0,1)^n$ the notation $\Pee_{\overline{p}}(.)$ will 
signify the situation where $X_1, \dots, X_n$ are independent random variables with 
$X_i \isd \Be(p_i)$.
Observe that, for every $A \subseteq \{0,1\}^2$, the probability
$\Pee_{\overline{p}}[ (X_1,\dots, X_n ) \in A ]$ can be written as a
polynomial in $p_1,\dots, p_n$.
In particular, this probability is a continuous function of the 
$p_i$-s and the partial derivatives $\frac{\partial}{\partial p_i}\Pee_{\overline{p}}{\big[} (X_1,\dots, X_n) \in A {\big]}$
exist.
Note also that if $A$ is an up-set then $\Pee_{\overline{p}}[ (X_1,\dots, X_n ) \in A ]$ is non-decreasing 
in each parameter $p_i$.

The following result is key to our proof of Theorem~\ref{thm:main}. It can be considered as an asymmetric version 
of Bourgain's powerful sharp threshold result (that appeared in the appendix of Friedgut's influential paper~\cite{Friedgut99}).

\begin{proposition}\label{prop:BourTobias}
 For every $C > 0$ and $0 < \alpha < 1/2$ there exist $K = K(C,\alpha) \in\eN, \delta = \delta(C,\alpha) > 0$ such that
 the following holds, for every $n \in \eN$ and every up-set $A \subseteq \{0,1\}^n$. \\
 If $\overline{p} \in (0,1)^n$ is such that $\Pee_{\overline{p}}{\big[ }(X_1,\dots, X_n) \in A {\big]} \in (\alpha,1-\alpha)$
and 
\[  \sum_{i=1}^n p_i(1-p_i) \cdot \frac{\partial}{\partial p_i}\Pee_{\overline{p}}{\big[} (X_1,\dots, X_n) \in A {\big]} \leq C, 
 \]%
%
%
then there exist indices $i_1, \dots, i_K \in \{1,\dots, n\}$ such that one of the following holds:
\begin{enumerate}
 \item[{\bf(a)}] $\Pee_{\overline{p}}{\big[} (X_1, \dots, X_n) \in A {\big|} X_{i_1} = \dots = X_{i_K} = 1 {\big]}
 \geq \Pee_{\overline{p}}{\big[} (X_1, \dots, X_n) \in A {\big]} + \delta$, or
 \item[{\bf(b)}] $\Pee_{\overline{p}}{\big[} (X_1, \dots, X_n) \in A {\big|} X_{i_1} = \dots = X_{i_K} = 0 {\big]}
 \leq \Pee_{\overline{p}}{\big[} (X_1, \dots, X_n) \in A {\big]} - \delta$.
\end{enumerate}
\end{proposition}

\noindent
What makes this result potentially very useful 
is the fact that $K$ and $\delta$ do not depend on the particular up-set or even the number of variables $n$.
Proposition~\ref{prop:BourTobias} can be derived in a relatively straightforward manner from a version of Bourgain's sharp threshold result for general 
probability spaces that can be found in O'Donnell's new book~\cite{ODonnellBoek}.
We defer the proof to Appendix~\ref{sec:BourApp}.

Recall that in the definition of the confetti model, we used a constant intensity Poisson process $\Pcal$ on 
$\eR^2\times (-\infty, 0]$. Throughout the paper, we will denote its intensity by $\lambda > 0$.
It follows from standard properties of the Poisson process (see for instance~\cite{KingmanBoek}) that the 
precise value of $\lambda$ is irrelevant: if we rescale the $z$-coordinates of the points of $\Pcal$ by a constant $a > 0$ then we 
obtain a Poisson process on $\eR^2 \times (-\infty, 0]$ with intensity $\lambda / a$. Since the vertical ordering of the
confetti disks is unchanged by this scaling, so are the colours that each of the points of the plane receives.

It is convenient to enumerate the points of our Poisson process $\Pcal$ as $\Pcal = \{ p_1, p_2, \dots \}$.
Furthermore, we let $C_i$ (for $i$-th confetti disk) denote the closed 
horizontal disk around $p_i$ of radius one and we let $D_i$ denote the projection of $C_i$ onto $\eR^2$ 
(recall that we identify $\eR^2$ with $\{z=0\} \subseteq \eR^3$).
The {\em visible part} of $D_i$ is defined as:

\[ V_i := D_i \setminus \bigcup_{j : z_j > z_i } \inter( D_j ). \]

\noindent
(Recall that $z_i$ is the last coordinate of $p_i$.)
Note that a visible part may consist of more than one path-connected component.
We will call the (path-) connected components of the visible parts {\em cells}.
The reader can probably easily convince him- or herself of the following straightforward fact, a formal proof of which can for instance be found 
in the work of Bordenave et al.~\cite{Bordenave2006}.

\begin{lemma}[\cite{Bordenave2006}, Lemma 2]\label{lem:bordenave}
Almost surely, every point of $\eR^2$ is contained in some cell, and every bounded set intersects only finitely many cells.
\end{lemma}

An elementary property of the Poisson process is that, almost surely, the set of all coordinates of its points 
will be algebraically independent (no subset is the a solution to a non-trivial polynomial equation with integer coefficients).
In particular, all $z$-coordinates are distinct, no two of the disks $D_i$ and $D_j$ are tangent, and no point lies on 
the boundary of more than two $D_i$s.
From this, together with Lemma~\ref{lem:bordenave}, it can be seen that almost surely: %

\begin{quote}
\begin{itemize}
\item[\CC{1}] Each cell has non-empty interior and is 
bounded by finitely many circle segments (this also includes the case where the boundary is a single circle), and;
\item[\CC{2}] Each point of the plane is either in the interior of some cell, on the boundary of exactly two cells
or on the boundary of exactly three cells. 
\end{itemize}%
\end{quote}%
See Figure~\ref{fig:dissection} for a depiction.
The points where three cells meet together with the circle segments separating adjacent cells can be viewed as an infinite three-regular plane graph.

\begin{figure}[h!]
\begin{center}
\includegraphics[trim=35 55 55 50,clip,width=7cm]{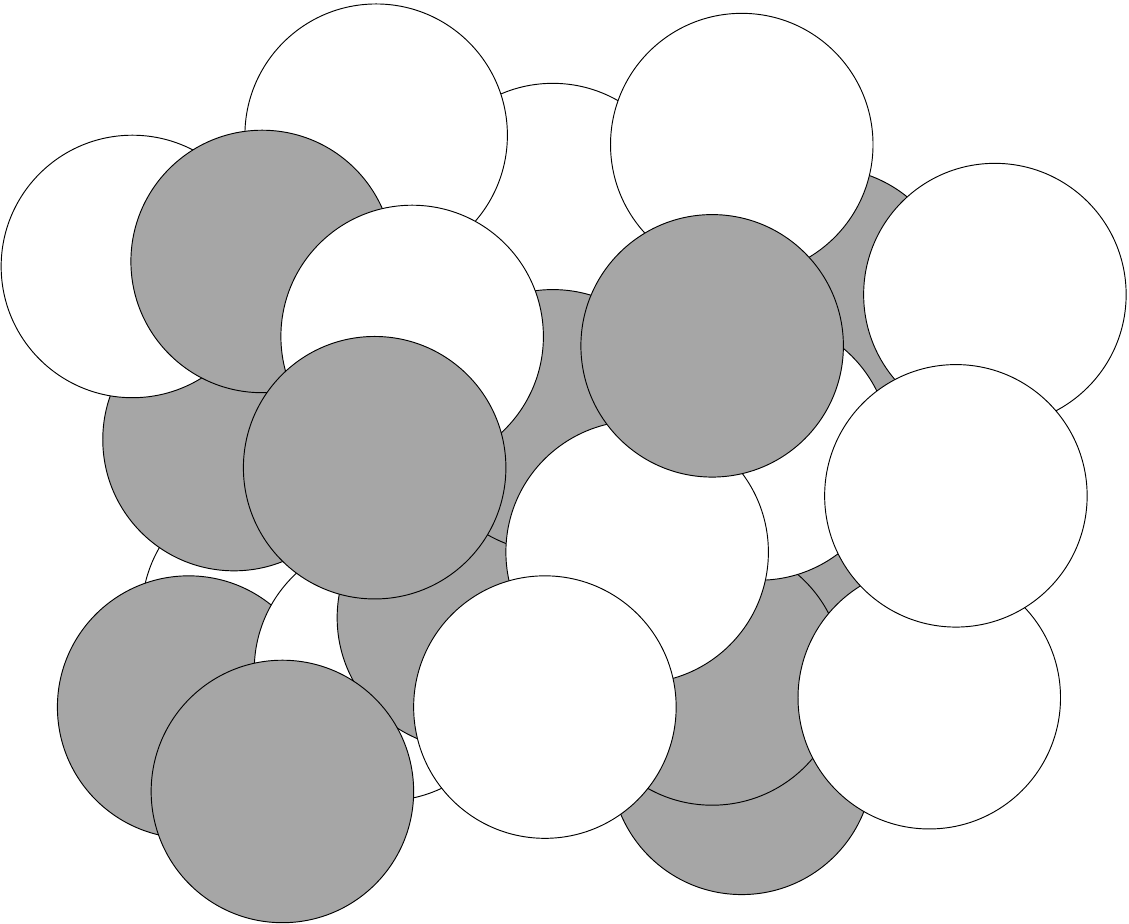}
\end{center}
\vspace{-.5cm}
\caption{A close-up of a realization of the confetti process. Each cell is bounded by finitely many circle
segments. There are points on the common boundary of three cells, but not of four.\label{fig:dissection}}
\end{figure}

%

Let $\Pb$ be the set of points of the Poisson process $\Pcal$ that receive a black disk, and let $\Pw \subseteq \Pcal$ denote those
points that receive a white disk. By standard properties of the Poisson process (see again~\cite{KingmanBoek}), $\Pb$ and $\Pw$ 
are {\em independent} Poisson processes with intensities $\lb := p \lambda$ and $\lw := (1-p)\lambda$, respectively, on $\eR^2 \times (-\infty,0]$.
Conversely, we can start with two independent Poisson processes $\Pb$ and $\Pw$ of intensities $\lb$ resp.~$\lw$ 
on $\eR^2 \times (-\infty,0]$. If we now center black disks on the points of $\Pb$ and white disks on the points of $\Pw$ 
then the situation is indistinguishable from the original setup with parameters $\lambda := \lb+\lw$ and 
$p := \lb / (\lb + \lw)$.
We will work with both settings in the paper, depending on which is more convenient at the time.
Sometimes we will use the notation $\Pee_{\lb, \lw}( . )$ to emphasize that we are working in the second 
setting (and to specify the values of $\lb, \lw$).

Formally speaking, we can say that the pair $(\Pb, \Pw)$ takes values in the set $\Omega$ whose elements are 
pairs $(\ob, \ow)$ of countable subsets of the lower halfspace $\eR^2 \times (-\infty, 0]$.
We will call a such pair $\omega = (\ob, \ow)$ a {\em configuration}.
A configuration specifies all the relevant information about a particular realization of the confetti model. 

We say that an event $E$ is {\em black increasing} if it is preserved under the addition of black points and the removal of white points.
That is, if we have a configuration $\omega = (\ob, \ow)$ for which $E$ holds, and we set $\ob' := \ob \cup A, 
\ow' := \ow \setminus B$ for arbitrary countable sets $A, B \subseteq \eR^2 \times (-\infty, 0]$ then
$E$ holds for the configuration $(\ob', \ow')$.
In this article we will rely heavily on the following generalization of Harris' inequality
due to Hirsch, which itself is based on a similar result of Bollob\'as and Riordan~\cite{BollobasRiordanVoronoi}.

\begin{lemma}[\cite{HirschArxiv}, Lemma 1]\label{lem:HarrisHirsch}
If $E_1, E_2$ are two black-increasing events then
\[ \Pee_p( E_1 \cap E_2 ) \geq \Pee_p(E_1) \Pee_p(E_2), \]
\noindent
for all $p \in [0,1]$.
\end{lemma}

Let $R \subseteq \eR^2$ be an axis-parallel rectangle.
We say that $R$ has a {\em black, horizontal crossing} if there is a polygonal curve
$\gamma \subseteq R$ between a point on the left edge of $R$ and a point of the right edge of $R$, such that
all points of $\gamma$ are black.
Similarly we say $R$ has a {\em black, vertical crossing} if there is such a curve between the bottom edge and the top edge of $R$, and
we define white horizontal and vertical crossings analogously.
Let us remark that the restriction to polygonal curves is not really a restriction at all: as can be seen from the earlier observations 
in this section, unless a certain event of probability zero holds, whenever there
is a black, continuous (but not necessarily polygonal) curve ``horizontally crossing'' $R$ then there also is a polygonal such curve.
By restricting attention to polygonal curves we avoid having to needlessly deal with topological intricacies in our proofs. 
In the rest of the paper we will write:

\[ \begin{array}{rcl}
H(R) & := & \{ \text{$R$ has a black, horizontal crossing}\}, \\
V(R) & := & \{ \text{$R$ has a black, vertical crossing}\}. 
\end{array} \]

\noindent
For notational convenience we will also write 

\[ 
H_{s\times t} := H( [0,s]\times[0,t] ). 
\]

A key ingredient to our proof of Theorem~\ref{thm:main} is the following result of Hirsch~\cite{HirschArxiv}, whose 
proof is essentially an adaptation of the sophisticated method developed by Bollob\'as and 
Riordan~\cite{BollobasRiordanVoronoi} to settle the critical probability for Voronoi percolation. 

\begin{theorem}[\cite{HirschArxiv}]\label{thm:HirschRSW} 
For every $\rho > 0$ we have that 
\[ \limsup_{s\to\infty} \Pee_{1/2}( H_{\rho s\times s} ) > 0. \]
\end{theorem}

\section{Discrete approximations to crossing events\label{sec:DiscrAppr}}

For $k \in \eN$, we dissect $[-k,k]^2\times [-k,0]$ into axis-parallel cubes of sidelength $2^{-k}$ in the obvious way.
We denote the collection of cubes obtained in this way, together with the set
$\left(\eR^2\times(-\infty,0]\right) \setminus \left([-k,k]^2\times [-k,0]\right)$, as $\Ccal_k$.
We say that a configuration $\omega \in\Omega$ 
is a {\em $k$-perturbation} of another configuration $\omega' \in \Omega$ 
if $c \cap \ob \neq \emptyset \Leftrightarrow c\cap \ob' \neq \emptyset$ and $c \cap \ow \neq \emptyset 
\Leftrightarrow c \cap \ow' \neq \emptyset$, for all $c\in\Ccal_k$.
For $E$ an arbitrary event and $k \in \eN$ we define the event $E^{(k)}$ as follows

\[ 
 E^{(k)} := \{ \omega \in \Omega :  
\text{ for every $k$-perturbation $\omega'$ of $\omega$, we have $\omega'\in E$}\} \]

\noindent
In other words, $E^{(k)}$ holds if $E$ holds for every configuration that can be obtained from the 
current realization of $(\Pb,\Pw)$ by wiggling, adding or removing points in such a way that the 
same parts of $\Ccal_k$ are hit by black, resp.~white, points.
Obviously we have, for every $k$ and every event $E$, that $E^{(k)} \subseteq E$.
Also note that $E^{(k)} \subseteq E^{(k+1)}$ for all $k$ since the partition $\Ccal_{k+1}$ refines $\Ccal_k$.

\begin{proposition}\label{prop:kmain}
For every $\eps > 0$, every $\lb,\lw>0$ and every bounded set $A \subseteq \eR^2$, there exists a $k_0 = k_0(\eps,\lb,\lw,A)$ such that 
\[ 
\sup_{R \subseteq A, \atop \text{$R$ axis-parallel rectangle}}%
{\Big |}\Pee_{\lb,\lw}[H(R)] - \Pee_{\lb,\lw}[ H^{(k)}(R) ]{\Big |} < \eps, 
\]
for all $k\geq k_0$. 
\end{proposition}

The rest of this section is devoted to the proof of the last proposition.
The proof is relatively straightforward and 
could be skipped in a first reading of the paper. 

For technical reasons it is convenient to also treat horizontal and vertical line segments and single points 
as rectangles in the remainder of this section. Of course, when $R$ is a vertical line segment, then $H(R)$ holds if at least one point of $R$ is 
black, and when $R$ is a horizontal line segment then $H(R)$ holds if all points of $R$ are black.

\begin{lemma}\label{lem:limk}
For every $\lb,\lw > 0$ and every axis-parallel rectangle $R$ we have that 
\[ \Pee_{\lb,\lw}[ H(R)] = \lim_{k\to\infty} \Pee_{\lb,\lw}[ H^{(k)}(R)]. \]
\end{lemma}

\begin{proof}
For notational convenience, let us write $E := H(R)$.
As already noted, we have $E^{(1)} \subseteq E^{(2)} \subseteq \dots \subseteq E$.
This gives

\[ \Pee_{\lb,\lw}(E) \geq \Pee_{\lb,\lw}\left( \bigcup_{k=1}^\infty E^{(k)} \right) = \lim_{k\to\infty} \Pee_{\lb,\lw}(E^{(k)}). \]

\noindent
It remains to show the reverse inequality.
To achieve this, it suffices show that for all configurations $\omega \in E$ except for a set of configurations of measure zero, 
we have $\omega \in \bigcup_{k=1}^\infty E^{(k)}$.

Let us thus fix an arbitrary configuration $\omega \in E$. 
It is convenient to enumerate $\ob\cup\ow$ as $\{p_1, p_2, \dots\}$ and to write $p_i = (x_i,y_i,z_i)$.
Discarding a set of configurations of total measure zero, we can assume without loss of generality that 
every bounded set contains finitely many points of $\ob\cup\ow$, that all the coordinates of all the $p_i$-s are distinct and that 
the properties~\CC{1} and~\CC{2} hold.
Hence, there is a black horizontal crossing $\gamma$ of $R$ that does not pass 
through any ``corners'' of cells (i.e.~points on the boundary  of three or more cells), and does not pass through any point on the 
common boundary of a black and a white cell.
Let us fix such a crossing $\gamma$.

Consider a point $q\in\gamma$. Let us suppose first that $q$ lies in the interior of some (black) cell.
That $q$ lies in the interior of a black cell means that the highest $p_i$ such that $\norm{q-\pi(p_i)} \leq 1$ belongs to $\ob$, and 
moreover $\norm{q-\pi(p_i)} < 1$.  (Here of course $\pi(x,y,z) := (x,y)$ denotes the projection onto the plane.)
From our assumptions on $\omega$, we see that there is a $\eps > 0$ such that
$\norm{q-\pi(p_i)} < 1-\eps$ and for every $j\neq i$ we have either $z_j < z_i - \eps$ or $\norm{q-\pi(p_j)} > 1+\eps$.
(Otherwise there are either two points with equal $z$-coordinates, or infinitely many points in some bounded region.)
Let us now fix an integer $k_0(q)$ satisfying $2^{-k_0(q)} < \eps/1000$ and $k_0(q) \geq  \norm{p_i}+1+\eps$.
Then we have that for every $k\geq k_0(q)$ and every $k$-perturbation $\omega'$ of $\omega$, the point $q$ together with 
all points of the plane at distance $<2^{-k_0(q)}$ from $q$ will be coloured black 
in the colouring of the plane defined by $\omega'$.

Suppose now that $q$ lies on the common boundary of two black cells (but not on a corner).
This means that if $p_i$, respectively $p_{j}$, are the highest, respectively second highest, points 
such that $\norm{q-\pi(p_i)},\norm{q-\pi(p_{j})} \leq 1$ then $p_{i}, p_{j} \in \ob$ and 
$\norm{q-\pi(p_i)}=1$ and $\norm{q-\pi(p_{j})} < 1$. 
Similarly to before, we see that there exist a $k_0(q)$ such that, for every $k\geq k_0(q)$ and every $k$-perturbation $\omega'$ 
of $\omega$, the point $q$ together with all points of the plane at distance $<2^{-k_0(q)}$ from $q$ will be coloured black under $\omega'$.

Next, let us observe that the disks $\{ B(q,2^{-k_0(q)}) : q \in \gamma \}$ form an open cover of the 
compact set $\gamma$. 
Hence there exists a finite subcover $\{ B(q_1,2^{-k_0(q_1)}), \dots, B(q_N,2^{-k_0(q_N)})\}$ that still covers $\gamma$.
Setting $k := \max( k_0(q_1), \dots, k_0(q_N) )$, we see that 
in every $k$-perturbation of $\omega$, the entire curve $\gamma$ is coloured black.
In other words, we have shown that every $\omega \in E$ except for a set of configurations of total measure zero lies in 
some $E^{(k)}$, as required. 
\end{proof}

\begin{lemma}\label{lem:cts2}
For every (fixed) $\lb, \lw > 0$ 
the probability $\Pee_{\lb,\lw}{\big[} H_{s\times t}{\big]}$ is continuous
as a function of $s,t$.
\end{lemma}

\begin{proof}
For $A \subseteq \eR^2$ let us write $x(A) := \sup\{ x : (x,y) \in A \}$ and let $\Kcal_t(A)$ denote the set of path-connected components of 
$A \cap \eR \times[0,t]$. We set $\Bcal_t := \bigcup\{ \Kcal_t(C) : \text{$C$ a black cell} \}$.
We point out that by \CC{1}, almost surely, $\Kcal_t(C)$ is finite for each cell $C$.
Hence, almost surely, $\Bcal_t$ is countable and each bounded set $A \subseteq \eR^2$ intersects finitely many elements of $\Bcal_t$
(using also Lemma~\ref{lem:bordenave} for the latter observation).

Let us now remark that $H_{s\times t} \subseteq H_{s'\times t}$ if $s' \leq s$.
Hence we have that 
$\Pee( H_{s\times t} ) - \lim_{s'\downarrow s} \Pee( H_{s'\times t} )
= \Pee\left( H_{s\times t} \setminus \bigcup_{s'>s} H_{s'\times t}\right)$.
We observe that if $H_{s\times t} \setminus \bigcup_{s'>s} H_{s'\times t}$ holds, then the rightmost 
point that can be reached by a black curve that stays inside $\eR \times [0,t]$ and starts in $\{s\}\times[0,t]$
must have $x$-coordinate exactly equal to $s$.
In particular there must be some $C \in \Bcal_t$ such that $x(C) = s$.
Since the colouring of the plane produced by the confetti percolation model is invariant
(in law) under horizontal translations and $\Bcal_t$ is almost surely countable, the probability 
that there exists such a $C$ equals zero. (Consider for instance the situation where we first generate the model in the usual way 
and then simultaneously translate all cells to the left by the same amount $U$, with $U$ uniform on $[0,1]$.)
This shows that $\lim_{s'\downarrow s} \Pee( H_{s'\times t} ) = \Pee( H_{s\times t} )$.

Similarly we have $\lim_{s'\uparrow s} \Pee( H_{s'\times t} ) - \Pee( H_{s\times t} ) = \Pee\left( \bigcup_{s'<s} H_{s'\times t} \setminus
H_{s\times t}\right)$. We observe that if $\bigcup_{s'<s} H_{s'\times t} \setminus
H_{s\times t}$ holds and only finitely many $C \in \Bcal_t$ intersect $[0,s]\times[0,t]$, then there must be some $C \in \Bcal_t$ with $x(C)=s$.
(Since only finitely many elements of $\Bcal_t$ intersect $[0,s]\times[0,t]$ and for every $s'<s$ there is a $C \in \Bcal_t$ with $s' \leq x(C) \leq s$.)
Hence we also have $\lim_{s'\uparrow s} \Pee( H_{s'\times t} ) = \Pee( H_{s\times t} )$, proving continuity in $s$.
Continuity in $t$ follows by the symmetry of the model.
\end{proof}



The final ingredient we will need for the proof of Proposition~\ref{prop:kmain} is the following variant of
Dini's theorem. 
It is an easy undergraduate exercise, but for completeness we provide a proof in Appendix~\ref{sec:dini}.

\begin{lemma}\label{lem:dini}
Let $I \subseteq \eR^d$ be an axis parallel box (in other words, a cartesian product of bounded, closed intervals) and let 
$f, f_1, f_2, \dots$ be functions satisfying:
\begin{enumerate}
 \item $f: I \to \eR$ is continuous, and;
 \item\label{itm:partii} $f_k : I \to \eR$ is non-decreasing (in each coordinate) for all $k\in\eN$, and;
 \item $f_{k+1}(x) \geq f_k(x)$ for all $k\in \eN, x\in I$, and;
 \item $\displaystyle\lim_{k\to\infty}f_k(x)=f(x)$ for all $x\in I$.  
\end{enumerate}
Then $f_k$ converges {\em uniformly} to $f$.
\end{lemma}

Let us remark that, by an easy reparametrization, the lemma also holds if instead of condition~\ref{itm:partii} 
$f_k$ is non-decreasing in some coordinates and non-increasing in the other coordinates -- 
with the set of coordinates on which $f_k$ is non-increasing being the same for each $k$.

\begin{proofof}{Proposition~\ref{prop:kmain}}
Without loss of generality we can take $A = [-K,K]^2$ for some $K$.
We define $f, f_1, f_2, \dots: [-K,K]^4 \to \eR$ by

\[ \begin{array}{l}
f(x_1,x_2,y_1,y_2) := \Pee_{\lb,\lw}{\Big[} H([x_1,\max(x_1,x_2)]\times[y_1,\max(y_1,y_2)]){\Big]}, \\
f_k(x_1,x_2,y_1,y_2) := \Pee_{\lb,\lw}{\Big[} H^{(k)}([x_1,\max(x_1,x_2)]\times[y_1,\max(y_1,y_2)]){\Big]}.
\end{array}
\]

\noindent
We have $f_{k+1} \geq f_k$ by definition of $E^{(k)}$.
By Lemma~\ref{lem:limk}, $f_k$ converges pointwise to $f$.
That $f$ is continuous follows from Lemma~\ref{lem:cts2} (note that
$f(x_1,x_2,y_1,y_2) = f(0,x_2-x_1,0,y_2-y_1)$).
Finally note that $f_k$ is non-decreasing in $x_1,y_1$ and non-increasing in $x_2,y_2$.
Appealing to Lemma~\ref{lem:dini} and the remark following it, we see that 
$f_k$ converges uniformly to $f$, which is precisely what Proposition~\ref{prop:kmain} states.
\end{proofof}

\section{Crossing probabilities when $p > 1/2$\label{sec:pGreaterHalf}}


\begin{proposition}\label{prop:cross} 
For every $p > 1/2$ it holds that $\displaystyle \sup_{s>1000} \Pee_p( H_{3s\times s} ) = 1$.
\end{proposition}

\begin{proof}
By Proposition~\ref{thm:HirschRSW}, there exists a constant $c>0$ and a sequence $(s_n)_n$
tending to infinity
such that $\Pee_{1/2}( H_{3s_n\times s_n} ) \geq c$ for all $n$.
By restricting to a subsequence if necessary, we can assume without loss of generality that 
$s_1 > 1000$ and $s_{i+1} > 1000 s_i$ for all $i$.

Let us fix an arbitrary $p > 1/2$ and $0<\alpha<c/2$. 
We will show that $\Pee_p( H_{3s_n\times s_n} ) \geq 1-\alpha$ for some $n\in\eN$, which will clearly
prove the proposition.
%

From now on we switch to the $\Pb, \Pw$ setting.
We pick 
$m = m(p,\alpha) \in \eN$ large (to be made precise later in the proof).
Appealing to Proposition~\ref{prop:kmain}, we can pick a $k=k(m)$ such that 

\begin{equation}\label{eq:choiceofk}
\left|\Pee_{1,1}[ H^{(k)}(R) ] - \Pee_{1,1}[ H(R) ]\right|, \left|\Pee_{1,1}[ V^{(k)}(R) ] -
\Pee_{1,1}[ V(R) ] \right| \leq  c/2,
\end{equation}

\noindent
for 
every axis-parallel rectangle
$R \subseteq [-1000s_m, 1000s_m]^2$. 
We now define the function $F:[0,1]\to[0,1]$ by:

\[ 
F(t) := \Pee_{\lb(t),\lw(t)}\left[ H^{(k)}_{3s_m\times s_m} \right],  
\]

\noindent 
where 

\[ 
\lb(t) := 1 + t(2p-1), \quad \lw(t) := 1 - t(2p-1). 
\]

\noindent
By the discussion in Section~\ref{sec:preliminaries}, we have 

\[ \begin{array}{c}
F(0) = \Pee_{1,1}\left[ H^{(k)}_{3s_m\times s_m} \right] = \Pee_{1/2}\left[ H^{(k)}_{3s_m\times s_m} \right], \\ 
F(1) = \Pee_{2p,2(1-p)}\left[ H^{(k)}_{3s_m\times s_m} \right] = \Pee_p\left[ H^{(k)}_{3s_m\times s_m} \right]. 
\end{array} \]

\noindent
In particular, $\Pee_p\left[ H_{3s_m\times s_m}\right] \geq F(1)$, so that it suffices to prove that $F(1) \geq 1-\alpha$.
Aiming for a contradiction, let us assume that $F(1) < 1-\alpha$ instead.

We now remark that, for every event $E$, whether or not $E^{(k)}$ holds is determined by a finite number of 
independent Bernoulli random variables. 
To make this more concrete, let us arbitrarily enumerate the side length $2^{-k}$-cubes of $\Ccal_k$ as $c_1,\dots,c_{n/2}$, 
where $n = k^3 2^{3k+1}$ is twice the number of such cubes.
We now define

\[ 
\begin{array}{rcl}
\text{For $1 \leq i\leq n/2$:} \quad
X_i & = & \left\{\begin{array}{cl}
  1 & \text{ if $c_i$ contains a black point, } \\
  0 & \text{ otherwise}. 
\end{array}\right., \\
& & \quad \\
\text{For $n/2< i\leq n$:} \quad
X_{i} & = & \left\{\begin{array}{cl}
  1 & \text{ if $c_{i-n/2}$ does {\em not} contain a white point, } \\
  0 & \text{ otherwise}. 
  \end{array}\right.
   \end{array}
\]

\noindent
Then the variables $X_1,\dots, X_n$ are independent Bernoulli random variables and we can write 
$\Pee_{\lb,\lw}\left[E^{(k)}\right] = \Pee[ (X_1,\dots,X_n) \in A ]$ for some
$A = A(E) \subseteq \{0,1\}^n$.

In particular, we can write 

\begin{equation}\label{eq:Frewrite}
 F(t) = \Pee_{\overline{p}(t)}[ (X_1,\dots,X_{n}) \in A ],
\end{equation}

\noindent
where $A \subseteq \{0,1\}^n$ is an up-set (as $H^{(k)}_{3s_m\times s_m}$ is a black-increasing event),
and the parameters $p_i(t) = \Ee X_i$ satisfy:

\[  
p_i(t) = 
\left\{
\begin{array}{cl}
1-\exp[-\lb(t) 2^{-3k} ] & \text{ for $1\leq i\leq n/2$, and } \\
\exp[-\lw(t) 2^{-3k}] & \text{ for $n/2 < i \leq n$.}
\end{array}\right.
\]

\noindent
The function $F$ is differentiable as can be seen from the expression~\eqref{eq:Frewrite} and the expressions for $p_i(t)$.
By the mean value theorem, there must be a $t \in [0,1]$ such that 

\[ F'(t) \leq F(1) - F(0) \leq 1. \] 

\noindent
By the chain rule we have

\[ F'(t) = \sum_{i=1}^n \frac{\partial}{\partial p_i} \Pee_{\overline{p}}[ (X_1,\dots,X_{n}) \in A ] \cdot p_i'(t). \]

\noindent
For $i \leq n/2$ we have 

\[ \begin{array}{rcl}
    p_i'(t) 
    & = &
    (2p-1) \cdot 
    2^{-3k} \cdot 
    \exp[-\lb(t)2^{-3k}] \\
    & \geq & 
    \left(\frac{2p-1}{2p}\right) \cdot 
    \lb(t) \cdot 
    2^{-3k} \cdot
    \exp[-\lb(t)2^{-3k}] \\
    & \geq & 
    \left(\frac{2p-1}{2p}\right) \cdot 
    (1-\exp[-\lb(t) 2^{-3k}]) \cdot
    \exp[-\lb(t)2^{-3k}] \\
    & = & 
    \left(\frac{2p-1}{2p}\right) p_i(t) (1-p_i(t)).
   \end{array}
\]

\noindent
(Here the second line follows since $\lb(t) \leq 2p$ and the third line follows since $e^{-x} \geq 1-x$ for all $x\geq 0$.)
Similarly, we find that for $i \geq n/2$:

\[ 
   p_i'(t)  = 
    (2p-1)\cdot 2^{-3k} \cdot \exp[-\lw(t)2^{-3k}] 
   \geq
    \left(\frac{2p-1}{2p}\right) p_i(t) (1-p_i(t)).
\]

\noindent
It follows that 

\[ 
 \sum_{i=1}^n p_i(1-p_i) \frac{\partial}{\partial p_i} \Pee_{\overline{p}}[ (X_1,\dots,X_{n}) \in A ] \leq
2p/(2p-1)
 =: C.
\]

\noindent
Using that $F$ is non-decreasing as $\lb(t)$ is increasing and $\lw(t)$ decreasing, that $F(1) < 1-\alpha$ by assumption, and that
$F(0) = \Pee_{1/2}[ H^{(k)}_{3s_m\times s_m}] \geq \Pee_{1/2}[ H_{3s_m\times s_m}] - c/2 \geq c/2 > \alpha$, 
we also have 

\[ 
 \Pee_{\opt}[ (X_1,\dots,X_{n}) \in A ] = F(t) \in [F(0),F(1)] \subseteq (\alpha, 1-\alpha).
\]

\noindent
Proposition~\ref{prop:BourTobias} 
thus provides us with indices $1\leq i_1, \dots, i_K \leq n$ and $b\in \{0,1\}$ such that  

\begin{equation}\label{eq:bourconc}
 \left|\Pee_{\opt}{\Big[} H_{3s_m\times s_m}^{(k)} {\Big|} X_{i_1}=\dots=X_{i_K}=b {\Big]} - 
 \Pee_{\opt}{\Big[} H_{3s_m\times s_m}^{(k)}{\Big]} \right| \geq \delta,
\end{equation}

\noindent
where $K = K(\alpha, C) \in \eN$ and $\delta = \delta(\alpha, C) > 0$ are constants that depend only on $p$ and $\alpha$ but 
-- and this is crucial for the current proof -- not on $m$ or $k$.

For $j=1,\dots, K$ let us fix a point $q_j \in \eR^2$ that is above the cube that $X_{i_j}$ corresponds to. (I.e.~$q_j$ is contained
in the projection of the cube $c_{i_j}$, respectively $c_{i_j-n/2}$, for $i_j \leq n/2$, respectively $i_j > n/2$.)
For $q \in \eR^2$ and $r>0$ let us denote 
$R_{q,r}^{\lf} := q+\left[\frac{-3r}{2},\frac{-r}{2}\right]\times\left[\frac{-3r}{2},\frac{3r}{2}\right], 
R_{q,r}^{\ri} := q + \left[\frac{r}{2},\frac{3r}{2}\right]\times\left[-\frac{3r}{2},\frac{3r}{2}\right],
R_{q,r}^{\tp} := q+\left[\frac{-3r}{2},\frac{3r}{2}\right]\times\left[\frac{r}{2},\frac{3r}{2}\right], 
R_{q,r}^{\bo} := q + \left[\frac{-3r}{2},\frac{3r}{2}\right]\times\left[\frac{-3r}{2},\frac{-r}{2}\right]$
and let $A_{q,r}$ denote the event $V(R_{q,r}^{\lf} ) \cap V(R_{q,r}^{\ri}) \cap H(R_{q,r}^{\tp})\cap H(R_{q,r}^{\bo})$. 
See Figure~\ref{fig:annulus} for a depiction.

\begin{figure}[h]
\begin{center}
 \input{annulus2.pspdftex}
 
 \vspace{-.7cm}
 
 \end{center}
 \caption{The event $A_{q,r}$.\label{fig:annulus}}
\end{figure}
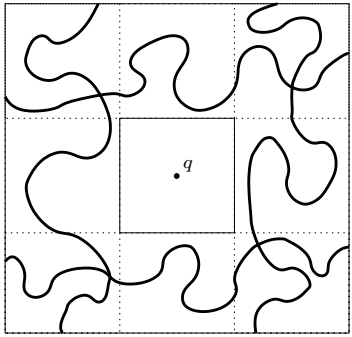

\noindent
Let us observe that the event $A_{q,r}$ implies 
that there is a closed, black, polygonal Jordan 
 curve that separates $q+\left(-\frac{r}{2},\frac{r}{2}\right)^2$ from $\eR^2 \setminus \left( q + \left[-\frac{3r}{2},\frac{3r}{2}\right]^2\right)$. 
For $j=1,\dots, K$ and $\ell=1,\dots,m-1$ let us define $E_{j,\ell} := A^{(k)}_{q_j,s_\ell}$.
By Lemma~\ref{lem:HarrisHirsch}, the choice of the sequence $(s_n)_n$ and~\eqref{eq:choiceofk}, we have that

\[ 
\Pee_{\opt}[ E_{j,\ell} ] \geq \Pee_{\op(0)}[ E_{j,\ell} ] \geq (c/2)^4. 
\]

\noindent
In the first inequality we also used that $E_{j,\ell}$ is a black-increasing event and 
$\opt \geq \op(0)$ as $\lb(t) \geq 1=\lb(0)=\lw(0)\geq \lw(t)$. \\
For $1\leq j \leq K$ we define $\Ical_j \subseteq \{1,\dots, m-1\}$ by:

\[ 
\ell \in \Ical_j 
\Leftrightarrow
\{ q_1,\dots, q_K \} \cap \left( q_{j} + 
\left[\frac{-3s_\ell-100}{2},\frac{3s_\ell+100}{2}\right]^2\setminus\left[\frac{-s_\ell+100}{2},\frac{s_\ell-100}{2}\right]^2 \right) 
= \emptyset. 
\]%
In particular, if $\ell \in \Ical_j$ then each of $q_1,\dots, q_K$ has distance at least 50 to the square annulus
$q_{j} + \left[\frac{-3s_\ell}{2},\frac{3s_\ell}{2}\right]^2\setminus\left[\frac{-s_\ell}{2},\frac{s_\ell}{2}\right]^2$.
Let us observe that, since $s_1\geq 1000, s_{i+1}\geq 1000 s_i$ by assumption, we have that
$|\Ical_j| \geq (m-1)-(K-1)=m-K$ for each $j$.
Let us set:

\[ E_j := \bigcup_{\ell \in \Ical_j} E_{j,\ell}. \]

\noindent
Using that 
the events $E_{j,1},\dots, E_{j,m-1}$ are independent, we find

\[ \Pee_{\opt}[ E_j ] =  
1 - \Pee_{\opt}
\Big{[}\bigcap_{\ell\in\Ical_j} E_{j,\ell}^c \Big{]}
\geq  1 - (1-(c/2)^4)^{|\Ical_j|} \geq 
1 - (1-(c/2)^4)^{m-K}. \]

\noindent 
Writing $E := \bigcap_{j=1}^K E_j$, we have:

\[ \Pee_{\opt}( E ) \geq 1 - K(1-(c/2)^4)^{m-K} \geq 1 - \delta/3, \]

\noindent
where the last inequality holds for $m$ sufficiently large.
Thus we also have that 

\begin{equation}\label{eq:delta1}
\Pee_{\opt}{\Big[} H^{(k)}_{3s_m\times s_m} \cap E {\Big]}  
\geq 
 \Pee_{\opt}{\Big[} H^{(k)}_{3s_m\times s_m} {\Big]} - \Pee_{\opt}{\Big[} E^c {\Big]} 
\geq 
 \Pee_{\opt}{\Big[} H^{(k)}_{3s_m\times s_m} {\Big]} - \delta/3.
\end{equation}

%

\noindent
Note that the event $E$ is independent of the event $\{ X_{i_1}=\dots=X_{i_K}=b \}$ since the 
state of these random variables can only influence the colour of points within distance less than two ($1+2^{1/2-k}$ to be exact) of $q_{1}, \dots, q_{K}$.
Hence, completely analogously to~\eqref{eq:delta1}, it follows that 

\begin{equation}\label{eq:delta2}
\Pee_{\opt}{\Big[} H^{(k)}_{3s_m\times s_m} \cap E {\Big|} X_{i_1} = \dots = X_{i_K} = b {\Big]} \geq
\Pee_{\opt}{\Big[} H^{(k)}_{3s_m\times s_m}{\Big|} X_{i_1} = \dots = X_{i_K} = b {\Big]}  - 
\delta/3
\end{equation}


\noindent
Next, we claim that:

\begin{claim}\label{clm:cruc} We have
$\Pee_{\opt}{\Big[} H^{(k)}_{3s_m\times s_m} \cap E {\Big|} X_{i_1} = \dots = X_{i_K} = b {\Big]}
=  \Pee_{\opt}{\Big[} H^{(k)}_{3s_m\times s_m} \cap E {\Big]}$.
\end{claim}

\begin{proofof}{Claim~\ref{clm:cruc}}
Let $B$ denote the event that there is a black, horizontal crossing of $R := [0,3s_m]\times[0,s_m]$ 
that does not get within distance two of any of the points $q_1,\dots,q_K$.
Obviously we have $B^{(k)} \subseteq H_{3s_m\times s_m}^{(k)}$.

We will show that $H_{3s_m\times s_m}^{(k)} \cap E \subseteq B^{(k)}$.
This implies that $H_{3s_m\times s_m}^{(k)} \cap E = B^{(k)} \cap E$ is an event that is 
independent of the state of $X_{i_1}, \dots, X_{i_K}$ (since these variables can only influence
the colours of points in the plane that are within distance two of $q_1,\dots,q_K$). This in turn implies the claim.

Let us thus pick a configuration $\omega \in H_{3s_m\times s_m}^{(k)} \cap E$, and consider an arbitrary $k$-perturbation $\omega'$ 
of $\omega$.
In the colouring of the plane defined by $\omega'$, there must be a black, horizontal crossing $\gamma$ of $R$ and
for each $j=1,\dots,K$ there is an $\ell_j \in \Ical_j$ and 
a black polygonal Jordan curve $\beta_j$ 
that separates  $q_j+\left(-s_{\ell_j}/2,s_{\ell_j}/2\right)^2$ from $\eR^2 \setminus \left(q_j+\left[-3s_{\ell_j}/2,3s_{\ell_j}/2\right]^2\right)$.

We will show that there is in fact a black, horizontal crossing $\gamma'$ of $R$ that does not come within distance $2$ of any of the
$q_i$-s.
To show this, it is enough to show that if $\gamma$ comes within distance $2$ of $q_j$, then there is a 
black, horizontal crossing $\gamma' \subseteq \gamma \cup \beta_j$ of $R$ that does not
come within distance 2 of $q_j$.
This is because $\gamma'$ will be within distance 2 of strictly fewer $q_i$-s than $\gamma$ is, as 
$\beta_j$ does not come within distance two of any of $q_1, \dots, q_K$ by choice of $\Ical_j$. 
We can thus apply induction on 
the number of $q_i$-s that are within distance two of $\gamma$ to find a crossing that does not come within distance
two of {\em any} $q_i$.

Let us first assume that $q_j+\left[-3s_{\ell_j}/2,3s_{\ell_j}/2\right]^2$ lies completely in $R$.
Observe that $\gamma$ must intersect $\beta_j$, since  $B(q_j,2) \subseteq q_j+\left[-s_{\ell_j}/2,s_{\ell_j}/2\right]^2$.
Let $z_1$, respectively $z_2$, be the first, respectively last, intersection point
of $\gamma$ and $\beta_j$. Here first and last refers to the order in which we encounter the intersection points as we traverse $\gamma$ from the
left side to the right side of $R$.  
Let $\gamma_1$ be the piece of $\gamma$ between the left side of $R$ and $z_1$ and let $\gamma_2$ be the piece of $\gamma$ between 
$z_2$ and the right side of $R$.
See Figure~\ref{fig:gamma12} for a depiction.

\begin{figure}[hbt]
\begin{center}
 \input{gamma6.pspdftex}
\end{center}
\caption{The curves $\gamma_1$ and $\gamma_2$ (not to scale).\label{fig:gamma12}}
\end{figure}
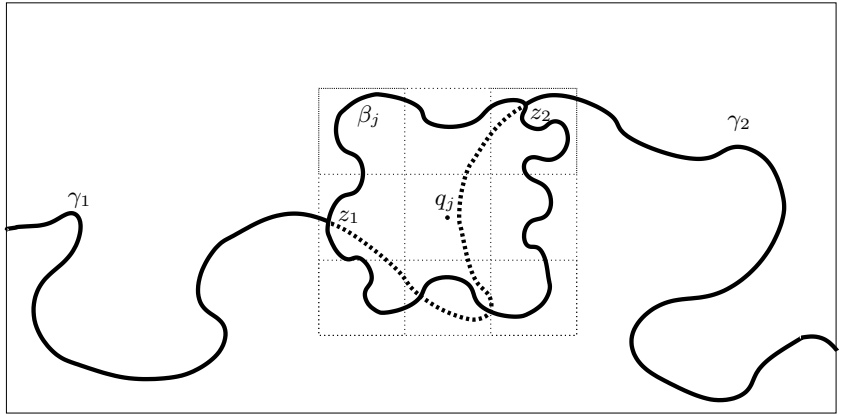

\noindent
Since $\beta_j$ is a polygonal Jordan curve, we can find a polygonal curve $\beta_j' \subseteq \beta_j$ between $z_1$ and $z_2$.
Clearly $\gamma' := \gamma_1 \cup \beta_j' \cup \gamma_2$ is a black, horizontal crossing of $R$ all of
whose points are at distance at least two from $q_j$. 

In the case when $q_j+\left[-3s_{\ell_j}/2,3s_{\ell_j}/2\right]^2$ intersects the boundary of $R$ 
we can clearly find a black horizontal crossing $\gamma' \subseteq \gamma \cup \beta_j$ that does not come within
distance 2 of $q_j$ in a similar manner. We leave the details to the reader. (Here it is important that $\beta_j$ cannot 
intersect opposite sides of $R$, since $s_{\ell_j}$ is at least a factor 1000 smaller than the height of $R$.)
Thus, by induction, there is a black horizontal crossing that does not come within distance two of
any of $q_1, \dots, q_K$.

We have now shown that $\omega' \in B$. 
As $\omega'$ is an arbitrary $k$-perturbation of $\omega$, we have $\omega \in B^{(k)}$ as required.
\end{proofof}

We now observe that Claim~\ref{clm:cruc} together with~\eqref{eq:delta1} and~\eqref{eq:delta2} implies
that

\[ \begin{array}{c} 
\left|\Pee_{\opt}{\Big[} H_{3s_m\times s_m}^{(k)} {\Big|} X_{i_1}=\dots=X_{i_K}=b {\Big]} - 
 \Pee_{\opt}{\Big[} H_{3s_m\times s_m}^{(k)}{\Big]} \right| \\
 \leq \\
  \left|\Pee_{\opt}{\Big[} H_{3s_m\times s_m}^{(k)} {\Big|} X_{i_1}=\dots=X_{i_K}=b {\Big]} - 
 \Pee_{\opt}{\Big[} H_{3s_m\times s_m}^{(k)} \cap E {\Big]} \right| \\
 +  \left|\Pee_{\opt}{\Big[} H_{3s_m\times s_m}^{(k)} \cap E {\Big]} - 
 \Pee_{\opt}{\Big[} H_{3s_m\times s_m}^{(k)}{\Big]} \right| \\
= \\
\left|\Pee_{\opt}{\Big[} H_{3s_m\times s_m}^{(k)} \cap E {\Big|} X_{i_1}=\dots=X_{i_K}=b {\Big]} - 
\Pee_{\opt}{\Big[} H_{3s_m\times s_m}^{(k)} {\Big|} X_{i_1}=\dots=X_{i_K}=b {\Big]} \right| \\
  +  \left|\Pee_{\opt}{\Big[} H_{3s_m\times s_m}^{(k)} \cap E {\Big]} - 
 \Pee_{\opt}{\Big[} H_{3s_m\times s_m}^{(k)}{\Big]} \right|   \\
 \leq \\
 2\delta/3, 
 \end{array} \] 
 
\noindent 
contradicting~\eqref{eq:bourconc}.
This contradiction proves that $F(1) \geq 1-\alpha$ as required.
%
%
\end{proof}

\section{The proof of Theorem~\ref{thm:main}\label{sec:grandfinale}}

That percolation does not occur (a.s.)~for $p \leq 1/2$ has already been proved by Hirsch~\cite{HirschArxiv}.
With Proposition~\ref{prop:cross} in hand, that percolation occurs (a.s.)~when $p > 1/2$ follows 
from a standard argument involving a comparison to 1-dependent percolation. 
(See for instance~\cite{BollobasRiordanBoek}, pages 73--75 and bottom of page 287.)

\section{Other confetti shapes.\label{sec:othershapes}}

Here we briefly describe the changes that need to be made in order for our proof to work for 
a more wide range of confetti shapes.
Theorem~\ref{thm:HirschRSW}, a key ingredient to our proof of Theorem~\ref{thm:main}, was in fact 
proved by Hirsch~\cite{HirschArxiv} under rather general assumptions on the shape of the confettis.
In the definition of the confetti process, instead of centering a horizontal disk on each point of the $\Pcal = \{p_1, p_2, \dots\}$, 
we can fix a ``shape'' $A \subseteq \eR^2$ and define collection $C_1, C_2, \dots$ by $C_i = p_i + A$, colour each $C_i$ black with probability 
$p$ and white with probability $1-p$ and then determine the colour of each point of the plane as before.
Let us consider the following list of axioms that we would like our confetti shape $A$ to satisfy:
\begin{quote}
\begin{itemize}
\item[\Ax{1}] $A$ is homeomorphic to the unit disk;
\item[\Ax{2}] $A$ is invariant under rotations by $\pi/2$ and reflections in the coordinate axes;
\item[\Ax{3}] $A$ is locally star-shaped, in the sense that for every $z\in A$ there is an open set $U \subseteq \eR^2$ such that
$z\in A\cap U$ and $A\cap U$ is star-shaped with center $z$;
\item[\Ax{4}] $\partial A$ is a Borel subset of $\eR^2$ with finite one-dimensional Hausdorff measure;
\item[\Ax{5}] The sets $\{ x \in \eR^2 : |\partial A \cap (x + \partial A)| = \infty \}$, 
$\{ x\in \eR^2 : A \cap (x+A) \text{ is not regular-open}\}$ 
and $\{ x \in \eR^2 : |(x+\partial A) \cap \{y=0\}| = \infty\}$ have Lebesgue measure zero.
\end{itemize}
\end{quote}
Our Theorem~\ref{thm:main} in fact generalizes to:

\begin{theorem}
If the confetti shape $A$ satisfies the axioms~\Ax{1}--\Ax{5}, then $p_c = 1/2$.
\end{theorem}

It is not hard to see that our set of axioms implies the more general set of axioms listed by Hirsch at the beginning of Section 2 
in~\cite{HirschArxiv}. In particular we know that Lemma~\ref{lem:HarrisHirsch}, Theorem~\ref{thm:HirschRSW} and 
$p_c \geq 1/2$ (Proposition 7 in~\cite{HirschArxiv}) hold in our setting.
We will now sketch the changes that need to be made to adapt our proof of $p_c \leq 1/2$ to the case of confetti shapes that
satisfy the axioms~\Ax{1}--\Ax{5}.

First, we remark that Lemma~\ref{lem:bordenave} also holds in our new setting with more general shapes. (In fact, Bordenave et al.~\cite{Bordenave2006} 
prove it in even greater generality.)
Property \CC{1} still holds if we substitute ``circle segment'' by ``segment of $\partial A$''. This follows 
from~\Ax{1} and~\Ax{5}.
Similarly, \CC{2} still holds by~\Ax{5}. 
(If there is a point on the boundary of more than three cells then 
either a) that point is simultaneously on the boundaries of three or more projected confetti leafs
or b) there are two projected confetti leafs that locally have only that point in common -- or more precisely, the intersection
of both projected leafs with an open disk around the point contains only that point.
Almost surely, no point is on the boundary of three or more projected confetti leafs since 
the first set in~\Ax{5} has measure zero. Almost surely, situation b) does not 
happen because the second set in~\Ax{5} has measure zero.)

We now see that the proof of Lemma~\ref{lem:limk} carries through verbatim if we make the substitutions:
$\norm{q-\pi(p_i)} \leq 1$ by $q \in \pi(p_i)+A$; $\norm{q-\pi(p_i)} < 1$ by $q \in \pi(p_i)+\inter A$;
$\norm{q-\pi(p_i)}<1-\eps$ by $B(q,\eps) \subseteq \pi(p_i)+A$; $\norm{q-\pi(p_j)}>1+\eps$ by $B(q,\eps) \subseteq \eR^2 \setminus(\pi(p_i)+A)$.
In proof of Lemma~\ref{lem:cts2} we used that (a.s.) $\Kcal_t(C)$ is finite for all $t$ and all cells $C$. 
That this is also the case in the current more general setting follows from the fact that the third set in \Ax{5} has Lebesgue measure zero.
The rest of the proof of Proposition~\ref{prop:kmain} then also carries through without any further changes.

In the proof of Proposition~\ref{prop:cross} the only things we may need to change are the constants 1000, 100, 50 and 2 
(the last constant occurring in the sentence after~\eqref{eq:delta1} and in the proof of Claim~\ref{clm:cruc}, as an upper bound on the distance needed 
between two points for their colours to be independent), to adjust for the new confetti 
shape $A$.  Multiplying these constants by $\max(\diam(A),1)$ will clearly do.

\subsection*{Acknowledgements}

I would like to thank Siamak Taati and Fabian Ziltener for helpful discussions related to the paper, and 
Ryan O'Donnell for helpful e-mail conversations. I thank the anonymous referees for corrections and suggestions that have 
improved the paper.

\bibliographystyle{plain}
\bibliography{ReferencesConfetti}

\appendix

\section{The proof of Proposition~\ref{prop:BourTobias}\label{sec:BourApp}}

Here we provide a proof of Proposition~\ref{prop:BourTobias}.
We will make use of an extension of Bourgain's sharp threshold result to general 
finite probability spaces that can be found in O'Donnell's new book~\cite{ODonnellBoek}.
Before we can state that result we need to give some more definitions.

Throughout the remainder of this section, $(V,\pi)$ will denote a finite probability space. That is, a finite set $V$ equipped with 
a probability measure $\pi$.
If $(V_1,\pi_1), \dots, (V_n,\pi_n)$ are such finite probability spaces, then we denote by
$\pi_1\otimes \dots \otimes \pi_n$ the probability distribution of the random vector $X = (X_1,\dots, X_n) \in V_1\times\dots\times V_n$ whose
coordinates are independent with $X_i \sim \pi_i$. 
If $\pi_1 = \dots = \pi_n = \pi$ then we also write $\pi^{\otimes n}$.
We define the {\em total influence} of a function $f : V_1\times\dots\times V_n \to \{0,1\}$ (wrt.~$\pi_1,\dots, \pi_n$) as

\[
 I(f) := \sum_{i=1}^n \Pee{\big[} f(X_1,\dots, X_n) \neq f(X_1,\dots,X_{i-1},Y_i,X_{i+1},\dots,X_n) {\big]}, 
\]

\noindent
where $(X_1,\dots,X_n),(Y_1,\dots, Y_n)\sim\pi_1\otimes\dots\otimes\pi_n$ are independent.
We remark that this definition of total influence differs from the definition of total influence used in several other texts 
such as~\cite{BollobasRiordanBoek}.

For $x \in V_1\times\dots\times V_n, T \subseteq \{1,\dots, n\}$, we denote by $x_T := (x_i)_{i\in T}$ the projection of $x$ onto the coordinates in $T$.
For $T \subseteq \{1,\dots, n\}$ and $\tau > 0$, we say that a vector $z \in \prod_{i\in T} V_i$ is a $\tau$-booster if 
$\Ee{\big[} f(X_1,\dots, X_n ) {\big|} X_i = z_i, \forall i \in T {\big]} \geq \Ee f(X_1,\dots,X_n) + \tau$.
For $\tau<0$, the vector $z \in \prod_{i\in T} V_i$ is a $\tau$-booster if instead 
$\Ee{\big[} f(X_1,\dots, X_n ) {\big|} X_i = z_i, \forall i \in T {\big]} \leq \Ee f(X_1,\dots,X_n) + \tau$.

We are now ready to present the generalized version of Bourgain's sharp threshold theorem.
The original version in~\cite{Friedgut99} was stated only for the case when $V = \{0,1\}$, but as noted in~\cite{ODonnellBoek}, the proof in fact 
generalizes to arbitrary $V$.
The following is a slight reformulation of the generalized version that can be found on page 303 of~\cite{ODonnellBoek}.

\begin{theorem}[\cite{ODonnellBoek}]\label{thm:Bourgain}
There exist absolute constants $c_1, c_2 > 0$ such that the following holds
for every finite set $V$, every probability measure $\pi$ on $V$, every integer $n \in \eN$ and every function $f : V^n \to \{0,1\}$.
If $\Var( f(X) ) \geq \frac{1}{400}$ then either 
\begin{itemize}
 \item[{\bf(a)}] $\Pee{\big[} \exists T \subseteq \{1,\dots,n\}, |T| \leq c_1 I(f), \text{ such that $X_T$ is a $\tau$-booster}{\big]}
 \geq \tau$, or
 \item[{\bf(b)}] $\Pee{\big[} \exists T \subseteq \{1,\dots,n\}, |T| \leq c_1 I(f), \text{ such that $X_T$ is a $(-\tau)$-booster}{\big]}
 \geq \tau$,
\end{itemize}
where $\tau = \exp[ -c_2 I^2(f) ]$ and $X = (X_1,\dots,X_n) \sim \pi^{\otimes n}$.
\end{theorem}

We have chosen to consider functions with values in $\{0,1\}$, while in~\cite{ODonnellBoek}
the theorem is stated in terms of $\{-1,1\}$-valued functions. 
It is however clear that if $f$ is $\{0,1\}$-valued then $g := 2f -1$ is $\{-1,1\}$-valued, 
$\Var(g) = 4\Var(f)$ and $x$ is a $\tau$-booster for $f$ if and only if it is a $2\tau$-booster
for $g$.
We should also mention that our definiton of total influence is different from the one given in~\cite{ODonnellBoek}. 
That the two definitions are equivalent follows Proposition 8.24 on page 204 of~\cite{ODonnellBoek}.

The last theorem also extends to asymmetric situations.

\begin{corollary}~\label{cor:asymBourgain}
There exist absolute constants $c_1, c_2 > 0$ such that the following holds
for all $n \in \eN$, all $(V_1, \pi_1), \dots, (V_n,\pi_n)$ and every function $f : V_1\times\dots\times V_n \to \{0,1\}$. \\
If $\Var(f) \geq \frac{1}{400}$ then either 
\begin{itemize}
 \item[{\bf(a)}] $\Pee{\big[} \exists T \subseteq \{1,\dots,n\}, |T| \leq c_1 I(f), \text{ such that $X_T$ is a $\tau$-booster}{\big]}
 \geq \tau$, or
 \item[{\bf(b)}] $\Pee{\big[} \exists T \subseteq \{1,\dots,n\}, |T| \leq c_1 I(f), \text{ such that $X_T$ is a $(-\tau)$-booster}{\big]}
 \geq \tau$,
\end{itemize}
where $\tau = \exp[ -c_2 I^2(f) ]$ and $X = (X_1,\dots,X_n)\sim\pi_1\otimes\dots\otimes\pi_n$.
\end{corollary}

\begin{proof}
We set $V := V_1\times\dots\times V_n, \pi = \pi_1\otimes\dots\otimes\pi_n$ and 
we define $g : V^n \to \{0,1\}$ by $g(z_1,\dots,z_n) := f( z_{1,1},\dots, z_{n,n} )$, 
where $z_{i,j}$ denotes the $j$-th coordinates of $z_i$.
Clearly, Theorem~\ref{thm:Bourgain} applies to this new situation.

If $Z_1, \dots, Z_n, Z_1', \dots, Z_n' \sim \pi$ are independent, then, for every $1\leq i \leq n$

\[ \begin{array}{c}
 \Pee{\big[} g(Z_1,\dots,Z_n) \neq g(Z_1,\dots,Z_{i-1},Z_i',Z_{i+1},\dots, Z_{n}){\big]} \\
  =  \\
 \Pee{\big[} f(Z_{1,1},\dots,Z_{n,n}) \neq f(Z_{11},\dots,Z_{i-1,i-1},Z_{i,i}',Z_{i+1,i+1},\dots, Z_{n,n}){\big]}. \\
\end{array} 
\]

\noindent
Since $(Z_{1,1},\dots,Z_{n,n}), (Z_{1,1}', \dots, Z_{n,n}') \sim \pi_1\otimes\dots\otimes\pi_n$ by construction, it follows that
$I(g) = I(f)$.
Similary, taking $X = (X_1,\dots,X_n)\sim\pi_1\otimes\dots\otimes\pi_n$, we find that $\Ee g(Z_1,\dots,Z_n) = \Ee f(X_1,\dots, X_n)$ and 

\[ \begin{array}{c}
\Pee{\big[} \exists T \subseteq \{1,\dots,n\}, |T| \leq c_1 I(g), \text{ such that $Z_T$ is a $\tau$-booster wrt.~$g$}{\big]} \\
= \\
\Pee{\big[} \exists T \subseteq \{1,\dots,n\}, |T| \leq c_1 I(f), \text{ such that $X_T$ is a $\tau$-booster wrt.~$f$}{\big]},
\end{array} \]

\noindent
where $c_1$ is as provided by Theorem~\ref{thm:Bourgain}.
This concludes the proof of the corollary.
\end{proof}

Let us note that the number $\frac{1}{400}$ in Theorem~\ref{thm:Bourgain} and Corollary~\ref{cor:asymBourgain}
could have been replaced by any other number (at the expense of changing the constants $c_1, c_2$). 
This can of course be seen from the proof of Theorem~\ref{thm:Bourgain}, but it can also be derived from the statement 
in a straightforward way. For completeness we prefer to spell out the details.

\begin{corollary}\label{cor:bouralpha}
For every $0 < \alpha < 1/2$ and $C>0$, there exist constants $K = K(\alpha,C) \in \eN, \delta = \delta(\alpha, C) > 0$ such 
that the following holds
for all $n \in \eN$, all $(V_1, \pi_1), \dots, (V_n,\pi_n)$ and every function $f : V_1\times\dots\times V_n \to \{0,1\}$. \\
If $\Pee[ f(X)=1 ] \in (\alpha,1-\alpha)$ and $I(f) \leq C$ then there exist indices $1\leq i_1,\dots, i_K \leq n$ and
values $x_{i_1} \in V_{i_1}, \dots, x_{i_K} \in V_{i_K}$ such that one of the following holds:
\begin{itemize}
 \item[{\bf(a)}] $\Pee{\big[} f(X) = 1 | X_{i_1} = x_{i_1}, \dots, X_{i_K}=x_{i_K} {\big]}
 \geq \Pee{\big[} f(X) = 1 {\big]} + \delta$, or
 \item[{\bf(b)}] $\Pee{\big[} f(X) = 1 | X_{i_1} = x_{i_1}, \dots, X_{i_K}=x_{i_K} {\big]}
 \leq \Pee{\big[} f(X) = 1 {\big]} - \delta$,
\end{itemize}
where $X = (X_1,\dots,X_n) \sim \pi_1\otimes\dots\otimes\pi_n$.
\end{corollary}

\begin{proof}
Let us first remark that $\Var(f) \leq \frac{1}{400}$ if and only if $\Ee f \in (\alpha_0, 1-\alpha_0)$, where
$\alpha_0 := (1-\sqrt{99/100})/2 \approx .0025$ is the smaller of the two solutions to $x(1-x)=\frac{1}{400}$.
Hence, if $\Ee f \in (\alpha_0,1-\alpha_0)$ then the result is an immediate
consequence of Corollary~\ref{cor:asymBourgain}, setting $K := c_1 \cdot C$ and $\delta := \exp[ - c_2 \cdot C^2 ]$.

Let us thus assume that ($\alpha < \alpha_0$ and) $\Ee f \in (\alpha,\alpha_0]\cup[1-\alpha_0,1-\alpha)$.
Switching to $g := 1-f$ if necessary (observe that this transformation leaves the total influence intact, and
{\bf(a)} holds for $f$ if and only if ${\bf(b)}$ holds for $g$, and similarly with $f,g$ switched), we can assume 
without loss of generality that $1-\alpha_0 \leq \Ee f < 1-\alpha$.

We now set $k :=  \lceil\log(1-\alpha_0)/\log(\Ee f)\rceil$ and $k_0 := \lceil \log(1-\alpha_0)/\log(1-\alpha)\rceil$ (Observe that
$k\leq k_0$ and $k_0$ depends only on $\alpha$ but not on $n, f$, the $V_i$-s or $\pi_i$-s.). 
We define

\[ g := \prod_{i=1}^k f(X_{(i-1)n+1},\dots,X_{in}), \] 

\noindent
where $X_1,\dots,X_{kn}$ are independent with $X_{jn+i} \sim \pi_i$ for all $0\leq j\leq k-1, 1\leq i \leq n$.
We have

\[ \Ee g = (\Ee f)^k \in (\alpha_0, 1-\alpha_0), \]

\noindent
so that Corollary~\ref{cor:asymBourgain} applies to $g$.
Hence there are indices $0 \leq i_1, \dots, i_{K} \leq kn$ and values $x_{i_1}, \dots, x_{i_K}$ such that 

\[ \left| \Ee{\big[} g | X_{i_1} = x_{i_1}, \dots, X_{i_K} = x_{i_K} {\big]} - \Ee g\right| \geq \tau, \]

\noindent
where $K := c_1\cdot C$ and $\tau = \exp[ -c_1\cdot C^2]$ with $c_1, c_2$ as provided by Corollary~\ref{cor:asymBourgain}.

Let us first suppose that $\Ee{\big[} g | X_{i_1} = x_{i_1}, \dots, X_{i_K} = x_{i_K} {\big]} \geq \Ee g + \tau$. 
We have

\[ \begin{array}{c}
\Ee{\big[} g(X_1,\dots,X_{kn}) | X_{i_1} = x_{i_1}, \dots, X_{i_K} = x_{i_K} {\big]} \\
 = \\
 \prod_{j=1}^k  \Ee{\big[} f( X_{(j-1)n+1},\dots, X_{jn}) | X_{i_1} = x_{i_1}, \dots, X_{i_K} = x_{i_K} {\big]}.
\end{array} \]

There must be a $ 1\leq j \leq k$ such that $\Ee{\big[} f( X_{(j-1)n+1},\dots, X_{jn}) | X_{i_1} = x_{i_1}, \dots, X_{i_K} = x_{i_K} {\big]}
\geq \Ee f + \frac{\tau}{k_02^{k_0}}$. This is because otherwise we would have

\[ \Ee{\big[} g | X_{i_1} = x_{i_1}, \dots, X_{i_K} = x_{i_K} {\big]} \leq 
\left(\Ee f + \frac{\tau}{k_02^{k_0}}\right)^k \leq (\Ee f)^k + \tau/2, \]

\noindent
using that $\frac{\dd}{\dd x}(\Ee f + x)^k$ is at most $k2^{k-1}$ for 
$x \in [0,1]$, so that $(\Ee f + x)^k \leq (\Ee f)^k + x k2^{k-1}$ for all $x\in[0,1]$.
Relabelling if necessary, we can assume without loss of generality that $j=1$.
Hence we have

\[ \Ee{\big[}f(X_1,\dots,X_n){\big |} X_{i_1}=x_{i_1},\dots. X_{i_K} = x_{i_K} {\big ]}
 \geq \Ee f(X_1,\dots, X_n) + \frac{\tau}{k_02^{k_0}}.
\]

\noindent
It may be that some indices $i_j$ are bigger than $n$. But in that case the value
$X_{i_j}$ is irrelevant to $f(X_1,\dots, X_n)$. To match into the framework of the Corollary we can just set $i_j = i_{j'}$ and 
$x_{i_j} = x_{i_{j'}}$ for some index $i_{j'} \leq n$.

Suppose then that $\Ee{\big[} g | X_{i_1} = x_{i_1}, \dots, X_{i_K} = x_{i_K} {\big]} \leq \Ee g - \tau$.
Similarly to before, there is a $ 1\leq j \leq k$ such that 
$\Ee{\big[} f( X_{(j-1)n+1},\dots, X_{jn}) | X_{i_1} = x_{i_1}, \dots, X_{i_K} = x_{i_K} {\big]}
\leq \Ee f - \frac{\tau}{k_02^{k_0}}$. We can continue as in the previous case.

This proves that the corollary indeed holds with $K = c_1\cdot C$ and $\delta = \exp[-c_2\cdot C^2]/k_02^{k_0}$.
\end{proof}

Our next ingredient is the well-known Margulis-Russo formula. 
A proof can for instance be found in~\cite{BollobasRiordanBoek}, where it appears as Lemma 9 on page 46.

\begin{lemma}[Margulis-Russo formula]
 Suppose that $f:\{0,1\}^n \to \{0,1\}$ is non-decreasing (coordinatewise). For $\overline{p} = (p_1,\dots,p_i)\in(0,1)^n$ we have
 
 \[ \frac{\partial}{\partial p_i} \Pee_{\overline{p}}[ f(X_1,\dots,X_n)=1 ]
  = \frac{1}{2p_i(1-p_i)}\Pee_{\overline{p}}[ f(X_1,\dots,X_n) \neq f(X_1,\dots,X_{i-1}, Y, X_{i+1},\dots,X_n) ], 
 \]

 \noindent
 where $X_1,\dots, X_n, Y$ are independent and $X_j \sim \Be(p_j)$ for all $1\leq j\leq n$ and $Y\sim\Be(p_i)$.
\end{lemma}

%
%
%
%
%
%
%
%

We now have all the ingredients for a quick proof of Proposition~\ref{prop:BourTobias}. For completeness we spell out the details.

\begin{proofof}{Proposition~\ref{prop:BourTobias}}
The up-set $A \subseteq \{0,1\}^n$ corresponds to a function $f : \{0,1\}^n \to \{0,1\}$ that is
one if and only if its input is in $A$. This function $f$ is non-decreasing since $A$ is an up-set.
By the Margulis-Russo formula we have that 

\[ I(f) = 2 \sum_{i=1}^n p_i(1-p_i) \cdot \frac{\partial}{\partial p_i}\Pee_{\overline{p}}{\big[} (X_1,\dots, X_n) \in A {\big]}. \]

\noindent
Also observe that, since $f$ is non-decreasing, we have 

\[ \begin{array}{c}
\Pee( f(X) = 1 | X_{i_1}=\dots =X_{i_K} = 1 ) \\
\geq \\
\Pee( f(X) = 1 | X_{i_1}=x_{i_1},\dots, X_{i_K} = x_{i_K} ) \\
\geq \\
\Pee( f(X) = 1 | X_{i_1}=\dots =X_{i_K} = 0 ), 
\end{array} \]

\noindent
for all $x_{i_1}, \dots, x_{i_K} \in \{0,1\}$.

Proposition~\ref{prop:BourTobias} is thus a direct corollary of Corollary~\ref{cor:bouralpha}.
\end{proofof}

\section{The proof of Lemma~\ref{lem:dini}\label{sec:dini}}

\begin{proofof}{Lemma~\ref{lem:dini}}
Let us fix $\eps>0$.
By continuity of $f$, for each $x\in I$ there is a set $U_x = [a_1,b_1]\times\dots\times[a_d,b_d] \subseteq I$ such that 
$x \in O_x := \inter_I( U_x )$ (we take the interior in the relative topology of $I$), and
$|f(x')-f(x)| < \eps/2$ for all $x'\in U_x$.
Since $f_k \to f$ pointwise and $f_{k+1} \geq f_k$, there is a $k_0(x)$ such that 
$f_k(a_1,\dots, a_d) \geq f(a_1,\dots, a_d) -\eps/2$ for all $k \geq k_0(x)$.
By the nondecreasingness of $f_k$ we also have that 
$f_k(x') \geq f(a_1,\dots,a_d) - \eps/2 \geq f(x')-\eps$ for all $x' \in U_x$ and all $k\geq k_0(x)$.

Since $I$ is compact and $\{ O_x : x\in I\}$ is an open cover of $I$, there must be a finite subcover
$\{ O_{x_1},\dots, O_{x_N}\}$ of $I$.
Setting $k_0 := \max( k_0(x_1), \dots, k_0(x_N) )$, it is clear that 
$f_k(x) \geq f(x)-\eps$ for every $x\in I$ and every $k\geq k_0$.
Also observe that $f_k(x) \leq f(x)$ for all $k$ and all $x \in I$, since
$f_k(x)$ is non-decreasing in $k$ and converges to $f(x)$ by assumption.
So we have that $|f_k(x)-f(x)| < \eps$ for all $x\in I$ and all $k \geq k_0$, which proves that 
$f_k$ converges uniformly as claimed.
\end{proofof}

\end{document}

%% file: annulus2.pspdftex
\begin{picture}(0,0)%
\includegraphics{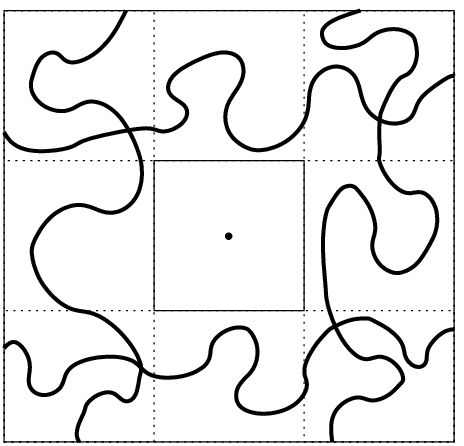}%
\end{picture}%
\setlength{\unitlength}{2368sp}%
\begingroup\makeatletter\ifx\SetFigFont\undefined%
\gdef\SetFigFont#1#2#3#4#5{%
  \reset@font\fontsize{#1}{#2pt}%
  \fontfamily{#3}\fontseries{#4}\fontshape{#5}%
  \selectfont}%
\fi\endgroup%
\begin{picture}(3666,3516)(1168,-3844)
\put(3063,-2086){\makebox(0,0)[lb]{\smash{{\SetFigFont{7}{8.4}{\rmdefault}{\mddefault}{\updefault}{\color[rgb]{0,0,0}$q$}%
}}}}
\end{picture}%

%% file: gamma6.pspdftex
\begin{picture}(0,0)%
\includegraphics{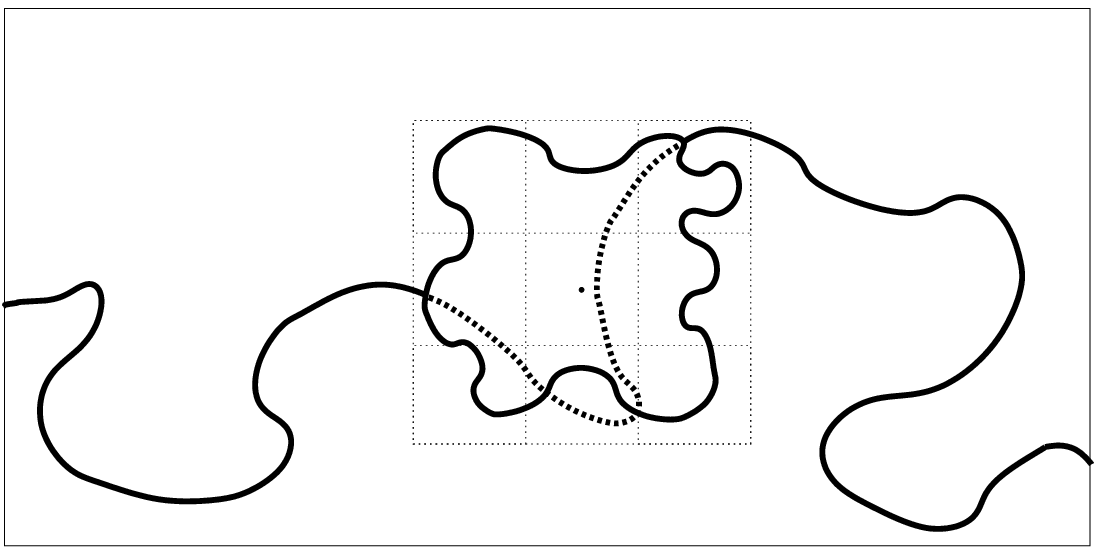}%
\end{picture}%
\setlength{\unitlength}{1776sp}%
\begingroup\makeatletter\ifx\SetFigFont\undefined%
\gdef\SetFigFont#1#2#3#4#5{%
  \reset@font\fontsize{#1}{#2pt}%
  \fontfamily{#3}\fontseries{#4}\fontshape{#5}%
  \selectfont}%
\fi\endgroup%
\begin{picture}(11698,5754)(-3208,-4908)
\put(2821,-1989){\makebox(0,0)[lb]{\smash{{\SetFigFont{9}{10.8}{\rmdefault}{\mddefault}{\updefault}{\color[rgb]{0,0,0}$q_j$}%
}}}}
\put(6903,-886){\makebox(0,0)[lb]{\smash{{\SetFigFont{9}{10.8}{\rmdefault}{\mddefault}{\updefault}{\color[rgb]{0,0,0}$\gamma_2$}%
}}}}
\put(-2302,-1961){\makebox(0,0)[lb]{\smash{{\SetFigFont{9}{10.8}{\rmdefault}{\mddefault}{\updefault}{\color[rgb]{0,0,0}$\gamma_1$}%
}}}}
\put(4143,-782){\makebox(0,0)[lb]{\smash{{\SetFigFont{9}{10.8}{\rmdefault}{\mddefault}{\updefault}{\color[rgb]{0,0,0}$z_2$}%
}}}}
\put(1467,-2207){\makebox(0,0)[lb]{\smash{{\SetFigFont{9}{10.8}{\rmdefault}{\mddefault}{\updefault}{\color[rgb]{0,0,0}$z_1$}%
}}}}
\put(1735,-823){\makebox(0,0)[lb]{\smash{{\SetFigFont{9}{10.8}{\rmdefault}{\mddefault}{\updefault}{\color[rgb]{0,0,0}$\beta_j$}%
}}}}
\end{picture}%